\documentclass[a4paper,oneside,11pt]{article}
\usepackage[a4paper, total={5.8in, 8.45in}]{geometry}

\title{Measure equivalence classification of right-angled Artin groups: the finite $\mathrm{Out}$ classes}
\author{Camille Horbez and Jingyin Huang}
\date{\today}

\usepackage{aeguill}                 
\usepackage{graphicx}
\usepackage{amsmath,amsfonts,amssymb,amsthm}
\usepackage{pstricks} 
\usepackage{epstopdf}
\usepackage{srcltx}
\usepackage[utf8]{inputenc} 
\usepackage{hyperref} 
\usepackage[english]{babel} 
\usepackage{comment}

\usepackage[all]{xy}
\usepackage{bbm}
\begin{document}
\newtheorem{de}{Definition}[section]
\newtheorem{dec}[de]{Definition-Construction}
\newtheorem{theo}[de]{Theorem} 
\newtheorem{prop}[de]{Proposition}
\newtheorem{lemma}[de]{Lemma}
\newtheorem{cor}[de]{Corollary}
\newtheorem{propd}[de]{Proposition-Definition}
\newtheorem{conj}[de]{Conjecture}
\newtheorem{claim}[de]{Claim}
\newtheorem*{claim2}{Claim}
\newtheorem{obs}[de]{Observation}
\newtheorem*{problem}{Problem}

\newtheorem{theointro}{Theorem}
\newtheorem*{defintro}{Definition}
\newtheorem{corintro}[theointro]{Corollary}

\theoremstyle{remark}
\newtheorem{rk}[de]{Remark}
\newtheorem{ex}[de]{Example}
\newtheorem{question}[de]{Question}
\newtheorem{assumption}[de]{Setting}

\normalsize

\newcommand{\st}{\mathrm{st}}
\newcommand{\lk}{\mathrm{lk}}
\newcommand{\cala}{\mathcal{A}}
\newcommand{\calb}{\mathcal{B}}
\newcommand{\calh}{\mathcal{H}}
\newcommand{\calg}{\mathcal{G}}
\newcommand{\calk}{\mathcal{K}}
\newcommand{\calp}{\mathcal{P}}
\newcommand{\calz}{\mathcal{Z}}
\newcommand{\caln}{\mathcal{N}}
\newcommand{\Out}{\operatorname{Out}}
\newcommand{\Aut}{\mathrm{Aut}}
\newcommand{\dunion}{\sqcup}
\newcommand{\B}{\mathcal B}

\newcommand{\calf}{\mathcal{F}}
\newcommand{\calq}{\mathcal{Q}}
\newcommand{\calc}{\mathcal{C}}
\newcommand{\calu}{\mathcal{U}}

\newcommand{\actson}{\curvearrowright}

\makeatletter
\edef\@tempa#1#2{\def#1{\mathaccent\string"\noexpand\accentclass@#2 }}
\@tempa\rond{017}
\makeatother

\maketitle

\begin{abstract}
    Given a right-angled Artin group $G$ with finite outer automorphism group, we determine which right-angled Artin groups are measure equivalent (or orbit equivalent) to $G$. 
\end{abstract}

\section{Introduction}

Right-angled Artin groups (RAAGs) form a particularly important class of groups, with a simple definition. Starting from a finite simple graph $\Gamma$ (i.e.\ with no loop-edges and no multiple edges between vertices), the right-angled Artin group with defining graph $\Gamma$ has one generator per vertex of $\Gamma$, and the relations are given by imposing that any two generators corresponding to adjacent vertices commute. We recommend \cite{Cha} for a general survey about these groups.  

Right-angled Artin groups have played a prominent role in geometric group theory, notably due to their central place in the theory of special $\mathrm{CAT}(0)$ cube complexes \cite{HW}. A lot of work has revolved around their quasi-isometric classification, e.g.\ \cite{BN,BKS,BJN,Hua,Hua2,Mar,Oh}.

This paper, a sequel to our previous work \cite{HH}, is concerned with the measure equivalence classification of right-angled Artin groups, a measurable analogue to the problem of their quasi-isometry classification. Recall from Gromov's work \cite[{Definition~0.5.$\text{E}_1$}]{Gro} 
that two countable groups $G,H$ are \emph{measure equivalent} if there exists a standard measure space $\Omega$, equipped with commuting measure-preserving free actions of $G$ and $H$ by Borel automorphisms, each having a Borel fundamental domain of finite positive measure. There is a finer notion, called \emph{orbit equivalence}, where one further imposes that the two fundamental domains for the actions of $G$ and $H$ on $\Omega$ coincide. See  \cite{Sha-survey,Gab-survey,Fur-survey} for surveys around these notions. 

In \cite{HH}, we proved that two RAAGs with finite outer automorphism groups are measure equivalent if and only if they are isomorphic -- this matches the quasi-isometry classification \cite{Hua}. In the present paper, we drop the finite $\Out$ condition on one of the two groups: given a right-angled Artin group $G$ with $|\Out(G)|<+\infty$, we classify all RAAGs (possibly with infinite outer automorphism group) that are measure equivalent to $G$. This time, as we will see, the measure equivalence and quasi-isometry classification differ. We mention that the condition that $|\Out(G)|<+\infty$ is easily readable from the defining graph of $G$, as follows from works of Servatius and Laurence \cite{Ser,Lau} reviewed in Section~\ref{sec:raag}.

Starting with a right-angled Artin group $G$ with defining graph $\Gamma_G$, there are two simple ways to build another right-angled Artin group $H$ which is measure equivalent to $G$, namely:
\begin{enumerate}
    \item Consider any graph product of infinite finitely generated free abelian groups over $\Gamma_G$. This is orbit equivalent to $G$, as follows from the combination of
    \begin{itemize}
    \item[--] a theorem of Dye, stating that all groups $\mathbb{Z}^n$ with $n\ge 1$ are orbit equivalent \cite{Dye1,Dye2} (later extended by Ornstein--Weiss to all countably infinite amenable groups \cite{OW}), 
    \item[--] the fact that orbit equivalence passes from the vertex groups to the graph product (\cite[Proposition~4]{HH}, building on \cite[{$\mathbf{P}_{\mathrm{ME}}\mathbf{6}$}]{Gab}).
    \end{itemize}
    \item Consider a finite-index RAAG subgroup of $G$. When $|\Out(G)|<+\infty$, the set of all such subgroups has been fully described by the second-named author in \cite[Section~6]{Hua}. In particular, they can be algorithmically computed in terms of their defining graphs. Here is an example, see e.g.\ \cite[Example~1.4]{BKS}: if $v\in V\Gamma_G$ is a vertex, and if $\theta:G\to\mathbb{Z}/k\mathbb{Z}$ is the homomorphism sending $v$ to $1$ and every other vertex to $0$, then the kernel of $\theta$ is a right-angled Artin group (of index $k$ in $G$), whose defining graph is obtained by gluing $k$ copies of $\Gamma_G$ along the star of $v$.  
\end{enumerate}

These two phenomena can be combined to generate more examples. In the present paper, we prove that these account for the full classification of right-angled Artin groups that are measure equivalent to $G$, when $|\Out(G)|<+\infty$.

\begin{theointro}[Measure equivalence classification, see Theorem~\ref{theo:me-classification}]\label{theointro:main}
Let $G$ and $H$ be two right-angled Artin groups. Assume that $|\Out(G)|<+\infty$. 

Then $G$ and $H$ are measure equivalent if and only if there exists a finite-index right-angled Artin subgroup $G^0\subseteq G$ such that, denoting by $\Lambda$ the defining graph of $G^0$, the group $H$ is a graph product of infinite finitely generated free abelian groups over $\Lambda$.
\end{theointro}

This time, the measure equivalence and quasi-isometry classifications differ. Under the assumptions of Theorem~\ref{theointro:main}, the groups $G$ and $H$ are quasi-isometric if and only if $H$ has finite index in $G$ by \cite[Theorem~1.2]{Hua}.

We also reach an orbit equivalence classification theorem, whose statement is even simpler. As follows from \cite[Proposition~5.9]{EH}, a right-angled Artin group $G$ with $|\Out(G)|<+\infty$ is never orbit equivalent to any proper finite-index subgroup of $G$. Therefore we reach the following statement.

\begin{theointro}[Orbit equivalence classification, see Theorem~\ref{theo:oe-classification}]\label{theointro:oe}
Let $G$ and $H$ be two right-angled Artin groups. Assume that $|\Out(G)|<+\infty$. 

Then $G$ and $H$ are orbit equivalent if and only if $H$ is a graph product of infinite finitely generated free abelian groups over the defining graph of $G$.
\end{theointro}

Let us now say a word about our proof strategy. From now on we will only consider RAAGs that are \emph{clique-reduced}, i.e.\ cannot be written as a graph product of free abelian groups where one of the vertex groups is isomorphic to $\mathbb{Z}^k$ for some $k\ge 2$. We do not lose any generality by doing so, since every RAAG is orbit equivalent to a clique-reduced one (using that a graph product of free abelian groups is orbit equivalent to the RAAG over the same graph). 

Kim and Koberda defined the \emph{extension graph} $\Gamma_G^e$ of $G$ as the graph whose vertices are the cyclic subgroups of $G$ associated to the vertices of its defining graph $\Gamma_G$, together with all their conjugates, where two cyclic subgroups are adjacent if they commute \cite{KK}. In our earlier work \cite[Theorem~2]{HH}, we proved that the extension graph is a measure equivalence invariant among \emph{transvection-free} RAAGs (which include those with $|\Out(G)|<+\infty$). In \cite[Corollary~9.4]{EH}, Escalier and the first-named author considered a variation over this graph, called the \emph{untransvectable extension graph}, and proved that it is a measure equivalence invariant among all clique-reduced RAAGs. It is the full subgraph of $\Gamma_G^e$ spanned by all \emph{untransvectable} parabolic subgroups, i.e.\ of the form $gG_vg^{-1}$ for some $g\in G$, and some $v\in V\Gamma_G$ such that there does not exist any vertex $w\neq v$ in $\Gamma_G$ with $\lk_{\Gamma_G}(v)\subseteq\st_{\Gamma_G}(w)$. The reader is refered to Section~\ref{sec:raag} for basic definitions and notations regarding parabolic subgroups. 

When $|\Out(G)|<+\infty$ and $H$ is clique-reduced, one could hope that an isomorphism between the untransvectable extension graphs of $G$ and $H$ forces $H$ to be a finite-index subgroup of $G$. Unfortunately this is not true in general, see Example~\ref{ex:counterexample}. But we prove that this is true under two extra assumptions regarding parabolic subgroups of $H$, namely (see Proposition~\ref{prop:rigidity}):

$\bullet$ the defining graph $\Gamma_H$ of $H$ does not contain any non-abelian untransvectable equivalence class (see Section~\ref{sec:raag} for definitions); 

$\bullet$ given a vertex $v\in V\Gamma_H$ with $H_v$ untransvectable, the common normalizer of all untransvectable cyclic parabolic subgroups that commute with $H_v$ is equal to $H_v$ itself.

These two conditions are automatically satisfied for $G$ with $|\Out(G)|<+\infty$, and we prove that they are both invariant under measure equivalence (Propositions~\ref{prop:non-abelian-class} and~\ref{prop:strongly-untransvectable}). So if a clique-reduced right-angled Artin group $H$ is measure equivalent to $G$, then it satisfies the above two conditions; from the isomorphism between the untransvectable extension graphs of $G$ and $H$ given by \cite{EH}, we then deduce that $H$ embeds as a finite-index subgroup of $G$, thereby proving our main theorem.

\paragraph{Acknowledgments.}

 We thank the referee for their prompt feedback and helpful comments on our paper.

The first-named author was funded by the European Union (ERC, Artin-Out-ME-OA, 101040507). Views and opinions expressed are however those of the authors only and do not necessarily reflect those of the European Union or the European Research Council. Neither the European Union nor the granting authority can be held responsible for them. The second-named author was funded by a Sloan fellowship.

This project was started at the Institut Henri Poincaré (UAR 839 CNRS-Sorbonne Université) during the trimester program \emph{Groups acting on fractals, Hyperbolicity and Self-similarity}. Both authors thank the IHP for its hospitality and support (through LabEx CARMIN, ANR-10-LABX-59-01).

\section{Background}\label{sec:background}

\subsection{Background on right-angled Artin groups}\label{sec:raag}

\paragraph{Right-angled Artin groups and graph products.} Let $\Gamma$ be a finite simple graph, i.e.\ $\Gamma$ has no loop-edge, and two distinct vertices are joined by at most one edge. We denote by $V\Gamma$ the vertex set of $\Gamma$. The \emph{right-angled Artin group} with defining graph $\Gamma$, denoted by $G_\Gamma$, is defined by the following presentation:
\[\langle V\Gamma \mid [v,w]=1 \text{~if~} v \text{~and~} w \text{~are joined by an edge} \rangle.\]
By a theorem of Droms \cite{Dro}, two right-angled Artin groups $G_\Gamma$ and $G_\Lambda$ are isomorphic if and only if the graphs $\Gamma$ and $\Lambda$ are isomorphic. Given a right-angled Artin group $G$, this will allow us to talk about the \emph{defining graph} of $G$, which will usually be denoted by  $\Gamma_G$.

Right-angled Artin groups belong to the more general class of \emph{graph products}, introduced by Green in \cite{Gre}. Given a finite simple graph $\Gamma$, and  a family $(G_v)_{v\in V\Gamma}$ of vertex groups, the graph product of the family $(G_v)$ over $\Gamma$ is the group obtained from the free product of the groups $G_v$, by adding as only extra relations that $G_v$ and $G_w$ commute whenever $v,w$ are adjacent in $\Gamma$. Right-angled Artin groups are graph products where all vertex groups are isomorphic to $\mathbb{Z}$.

\paragraph*{Parabolic subgroups.} Recall that a subgraph $\Lambda\subseteq\Gamma$ is a \textit{full subgraph} if any two vertices of $\Lambda$ are adjacent in $\Lambda$ if and only if they are adjacent in $\Gamma$. Each full subgraph $\Lambda\subseteq\Gamma$ gives rise to an embedding $G_{\Lambda}\hookrightarrow G_\Gamma$.
A \emph{parabolic subgroup} of $G_\Gamma$ is a subgroup $P$ which is conjugate to $G_{\Lambda}$ for some full subgraph $\Lambda\subseteq\Gamma$; in this case $\Lambda$ is unique and is called the \emph{type} of $P$ -- see e.g.\ \cite[Proposition~2.2]{CCV}. Intersections of parabolic subgroups of  $G_\Gamma$ are again parabolic subgroups \cite[Proposition~2.8]{DKR}. We say that a parabolic subgroup is \emph{standard} if it is equal (and not just conjugate) to $G_{\Lambda}$ for some full subgraph $\Lambda\subseteq\Gamma$.

The \emph{star} of a vertex $v$ in $\Gamma$, denoted by $\st_{\Gamma}(v)$, is the full subgraph spanned by $v$ and all the vertices that are adjacent to $v$. Its \emph{link} $\lk_{\Gamma}(v)$ is the full subgraph spanned by all the vertices that are adjacent to $v$. Given a full subgraph $\Lambda\subseteq \Gamma$, we let $\Lambda^{\perp}$ be the full subgraph of $\Gamma$ spanned by the vertices in $V\Gamma\setminus V\Lambda$ that are adjacent to all vertices of $\Lambda$. For example, for every vertex $v\in V\Gamma$, we have $\lk(v)=\{v\}^\perp$. 

By \cite{God} or \cite[Proposition~2.2]{CCV}, if $\Lambda\subseteq\Gamma$ is a full subgraph, then the normalizer of $G_{\Lambda}$ in $G_\Gamma$ is $G_{\Lambda}\times G_{\Lambda^{\perp}}$. Given a parabolic subgroup $P=gG_\Lambda g^{-1}$ of $G_\Gamma$, we let $P^{\perp}=gG_{\Lambda^{\perp}}g^{-1}$, which is well defined (i.e.\ it does not depend on the choice of $g$) in view of the description of normalizers. The normalizer of $P$ is then $P\times P^{\perp}$.

The center of a parabolic subgroup is a parabolic subgroup \cite[Proposition~2.2(1)]{CCV}: for an induced subgraph $\Lambda\subseteq\Gamma$, the center of $G_\Lambda$ is equal to $G_{\Theta}$, where $\Theta\subseteq\Lambda$ is the subgraph spanned by all vertices that are adjacent to all other vertices of $\Lambda$. 

A \emph{product parabolic subgroup} of $G_\Gamma$ is a parabolic subgroup that is conjugate to $G_\Lambda$ for some full subgraph $\Lambda\subseteq\Gamma$ that splits as a join of two non-empty subgraphs $\Lambda=\Lambda_1\circ\Lambda_2$ (i.e.\ every vertex of $\Lambda_1$ is adjacent to every vertex of $\Lambda_2$). In this case $G_\Lambda$ splits as a direct product $G_{\Lambda}=G_{\Lambda_1}\times G_{\Lambda_2}$. Every product parabolic subgroup is contained in one which is maximal for inclusion. In addition, a standard product parabolic subgroup is always contained in a maximal one which is also standard (this is a consequence of \cite[Proposition~2.2(2)]{CCV}). 

A parabolic subgroup that is not a product parabolic subgroup is called \emph{irreducible}. Every full subgraph $\Lambda$ of $\Gamma$ has a unique (up to permutation of the factors) join decomposition of the form $\Lambda=\Lambda_1\circ\dots\circ\Lambda_k$, where none of the subgraphs splits non-trivially as a join. Then $G_{\Lambda}=G_{\Lambda_1}\times\dots\times G_{\Lambda_k}$. The subgroups $G_{\Lambda_i}$ are called the \emph{irreducible factors} of $G_{\Lambda}$. More generally, if $P=gG_{\Lambda}g^{-1}$, then the subgroups $gG_{\Lambda_i}g^{-1}$ are called the \emph{irreducible factors} of $P$ (that these are well-defined, i.e.\ they do not depend on the choice of $g$, follows from the description of normalizers in $G_\Gamma$).  

\paragraph*{Collapsible parabolic subgroups.} We borrow the following definition from \cite[Definition~6.1 and Section~9]{EH}.

\begin{de}[Collapsible, clique-reduced]\label{de:collapsible}
Let $\Gamma$ be a finite simple graph. A full subgraph $\Lambda\subseteq\Gamma$ is \emph{collapsible} if for any two vertices $x,y\in V\Lambda$, one has $\st_{\Gamma}(x)\cap (\Gamma\setminus\Lambda)=\st_{\Gamma}(y)\cap(\Gamma\setminus\Lambda)$. A parabolic subgroup of $G_{\Gamma}$ is \emph{collapsible} if its type is collapsible.

We say that $\Gamma$ is \emph{clique-reduced} if it does not contain any collapsible complete subgraph on at least two vertices. We also say that a right-angled Artin group is \emph{clique-reduced} if its defining graph is clique-reduced.
\end{de}

\begin{rk}\label{rk:clique-reduced}
Collapsibility is interpreted as follows. Assume that $\Lambda$ is collapsible, and let $\bar\Gamma$ be the graph obtained from $\Gamma$ by collapsing $\Lambda$ to a single vertex $v_\Lambda$, and joining $v_{\Lambda}$ by an edge to every vertex $w\in V\Gamma\setminus V\Lambda$ that was adjacent to some (equivalently, any) vertex of $\Lambda$. Then $G_{\Gamma}$ is isomorphic to the graph product over the graph $\bar\Gamma$, where the vertex group of $v_{\Lambda}$ is $G_{\Lambda}$, and all other vertex groups are $\mathbb{Z}$.  

In particular, every right-angled Artin group can be written as a graph product of free abelian groups over a clique-reduced graph.
\end{rk}

The following lemma gives other characterizations of collapsibility. In the sequel, we will only use the equivalence of the first two assertions. The equivalence with (3) will be proved in the more general groupoid-theoretic setting in Lemma~\ref{lemma:coll}; working out the group-theoretic version first serves as a warm-up.

\begin{lemma}\label{lemma:collapsible}
Let $G$ be a right-angled Artin group, and let $P\subseteq G$ be a parabolic subgroup. The following statements are equivalent.
\begin{itemize}
    \item[(1)] The parabolic group $P$ is collapsible.
    \item[(2)] For every infinite parabolic subgroup $B\subseteq P$, one has $B\times B^{\perp}\subseteq P\times P^{\perp}$.
    \item[(3)] For every infinite subgroup $B\subseteq P$, the normalizer of $B$ is contained in the normalizer of $P$.
\end{itemize}
\end{lemma}

\begin{proof}
Up to conjugation, we will assume that $P=G_\Lambda$ for some full subgraph $\Lambda\subseteq\Gamma_G$.

We first prove that $(1)\Rightarrow (2)$, so assume that $P$ (equivalently $\Lambda$) is collapsible, and let $B\subseteq P$ be an infinite parabolic subgroup. Up to a conjugation by an element in $P$, we can assume that $B=G_{\Theta}$ for some full subgraph $\Theta\subseteq\Lambda$, see \cite[Proposition~2.2(2)]{CCV}. Collapsibility of $\Lambda$ implies that $\Theta^{\perp}\subseteq\Lambda\circ\Lambda^{\perp}$, so $B\times B^{\perp}\subseteq P\times P^{\perp}$, i.e.\ (2) holds.

We now prove that $\neg (1)\Rightarrow \neg (2)$, so assume that $P$ (equivalently $\Lambda$) is not collapsible. Then $\Lambda$ contains a vertex $x$ such that $\lk_{\Gamma}(x)\nsubseteq \Lambda\circ \Lambda^{\perp}$. Letting $B$ be the cyclic parabolic subgroup associated to $x$, we thus have $B\times B^{\perp}\nsubseteq P\times P^{\perp}$, showing that (2) fails. 

The implication $(3)\Rightarrow (2)$ is clear, because $B\times B^{\perp}$ is the normalizer of $B$ when $B$ is a parabolic subgroup. For $(2)\Rightarrow (3)$, start with an infinite subgroup $B\subseteq P$. Let $\hat B$ be the intersection of all parabolic subgroups that contain $B$. In particular $\hat B$ is a parabolic subgroup contained in $P$. By (2) we have $\hat B\times \hat B^{\perp}\subseteq P\times P^{\perp}$. Since the normalizer of $B$ is contained in the normalizer of $\hat B$ (i.e.\ $\hat B\times \hat B^{\perp}$), Assertion~(3) follows.
\end{proof}

\paragraph*{Automorphisms of right-angled Artin groups.} Consider the standard generating set of $G_\Gamma$ given by $V\Gamma$. By work of Laurence \cite{Lau}, confirming a conjecture of Servatius \cite{Ser}, the group $\Out(G_\Gamma)$ is generated by the outer classes of the following automorphisms:
\begin{itemize}
     \item \emph{graph automorphisms}, i.e.\ permutations of the generating set $V\Gamma$ according to an automorphism of $\Gamma$;
    \item \emph{inversions}, i.e.\ sending one generator $s$ to $s^{-1}$, keeping all other generators fixed;
    \item \emph{transvections}: given two distinct generators $s_1,s_2$ associated with vertices $v_1,v_2\in V\Gamma$ with $\lk_{\Gamma}(v_1)\subseteq\st_{\Gamma}(v_2)$, send $s_1$ to $s_1s_2$, keeping all other generators fixed;
    \item \emph{partial conjugations}: given a generator $s$ associated with a vertex $v\in V\Gamma$ whose star disconnects $\Gamma$, and given a connected component $C$ of $\Gamma\setminus\st_{\Gamma}(v)$, conjugate all generators associated with vertices in $C$ by $s$, and keep all other generators fixed.
\end{itemize}
The first two types of automorphisms generate a finite subgroup of $\Aut(G_\Gamma)$. Thus $|\Out(G_\Gamma)|<+\infty$ if and only if there are no transvections and no partial conjugations. We say that $G_\Gamma$ is \emph{transvection-free} if it has no transvections.

\label{equivalence-classes}In \cite[Section~2]{CV}, Charney and Vogtmann introduced the following preorder on $V\Gamma$: say that $v\le w$ if $\lk_\Gamma(v)\subseteq\st_{\Gamma}(w)$. Say that $v$ and $w$ are \emph{equivalent}, denoted by $v\sim w$, if $v\le w$ and $w\le v$. Equivalence classes have two possible forms: either all vertices in the class are pairwise adjacent, or no two vertices in the class are adjacent \cite[Lemma~2.3]{CV}. Equivalence classes are called \emph{abelian} or \emph{non-abelian}, respectively (the corresponding parabolic subgroups are free abelian or non-abelian free, respectively). We will write $v<w$ to mean $v\le w$ and $v\not\sim w$. Observe that if $\Theta$ is non-abelian equivalence class of vertices, then for every $v\in \Theta$, we have $\Theta^{\perp}=\lk_{\Gamma}(v)$.

We say that a vertex $v\in V\Gamma$ is \emph{transvectable} if there exists a vertex $w\in V\Gamma$ distinct from $v$ such that $\lk_{\Gamma}(v)\subseteq\st_\Gamma(w)$, and \emph{untransvectable} otherwise. More generally, we will say that a full subgraph $\Lambda\subseteq\Gamma$ is \emph{transvectable} if there exists a vertex $w\in V\Gamma\setminus V\Lambda$ such that for every $v\in V\Lambda$, one has $\lk_{\Gamma}(v)\subseteq\st_{\Gamma}(w)$ (otherwise it is \emph{untransvectable}). A parabolic subgroup of $G_\Gamma$ is \emph{transvectable} (or \emph{untransvectable}) if its type is. 

In other words, a vertex is untransvectable if it is maximal for the preorder $\le$, and its equivalence class is reduced to itself. Notice also that an equivalence class of vertices (for $\sim$) is untransvectable if and only if it is maximal for the induced partial order on equivalence classes of vertices. 

\begin{rk}\label{rk:na}
 A parabolic subgroup $P$ of $G_\Gamma$ is conjugate to $G_{\Lambda}$ for some untransvectable non-abelian equivalence class $\Lambda\subseteq\Gamma$ if and only if it is collapsible, non-abelian free, and untransvectable.
\end{rk}

\subsection{Background on measured groupoids}\label{sec:groupoids}

In this section, we review the language of measured groupoids, in which the proof of our measure equivalence classification theorem is phrased. General references of relevance for the present paper include \cite[Section~2.1]{AD} or \cite{Kid-survey}; we will follow the terminology and presentation from \cite[Section~3]{GH} or \cite[Section~4]{EH} very closely. The relevance of measured groupoids for studying measure equivalence is explained in the final paragraph of this section.

\paragraph{Generalities on measured groupoids.} Let $X$ be a standard Borel space, i.e.\ $X$ is the Borel space associated to a Polish (i.e.\ separable and completely metrizable) topological space. A \emph{discrete Borel groupoid} over $X$ is a standard Borel space $\calg$ (thought of as a space of arrows) that is equipped with two Borel maps $s,r:\calg\to X$ (giving the source and range of the arrows) whose fibers are at most countable, with a measurable composition law $(\mathsf{g},\mathsf{h})\mapsto \mathsf{gh}$ defined on the set of pairs $(\mathsf{g},\mathsf{h})$ such that $s(\mathsf{g})=r(\mathsf{h})$, a measurable inverse map $\mathsf{g}\mapsto\mathsf{g}^{-1}$ (with $s(\mathsf{g}^{-1})=r(\mathsf{g})$ and $r(\mathsf{g}^{-1})=s(\mathsf{g})$), and a unit element $\mathsf{e}_x$ for every $x\in X$ (with $\mathsf{g}\mathsf{g}^{-1}=\mathsf{e}_{r(\mathsf{g})}$ and $\mathsf{g}^{-1}\mathsf{g}=\mathsf{e}_{s(\mathsf{g})}$ for every $\mathsf{g}\in\mathcal{G}$).

A discrete Borel groupoid is \emph{trivial} if all its elements are units.

A \emph{bisection} of $\calg$ is a Borel subset $B\subseteq\calg$ such that $s_{|B}$ and $r_{|B}$ are injective. By a theorem of Lusin and Novikov (see e.g.\ \cite[Theorem~18.10]{Kec}), every discrete Borel groupoid is covered by countably many bisections.

When $X$ is equipped with a Borel probability measure $\mu$, we say that $\calg$ is a \emph{measured groupoid} over $(X,\mu)$ if additionally, the measure $\mu$ is quasi-invariant, i.e.\ for every bisection $B\subseteq\calg$, one has $\mu(r(B))=0$ if and only if $\mu(s(B))=0$. From now on $\calg$ will always be a measured groupoid over $X$.

A \emph{measured subgroupoid} of $\calg$ is a measurable subset $\calh\subseteq\calg$ such that for every $\mathsf{h}_1,\mathsf{h}_2\in\calh$ with $s(\mathsf{h}_1)=r(\mathsf{h}_2)$, one has $\mathsf{h}_1\mathsf{h}_2\in\calh$.

Given a measurable subset $U\subseteq X$, the \emph{restricted groupoid} $\calg_{|U}$ is the measured groupoid over $U$ consisting of all elements $\mathsf{g}\in\calg$ with $s(\mathsf{g}),r(\mathsf{g})\in U$.

The following is an important example of a measured groupoid. Let $X$ be a standard probability space, equipped with a measure-preserving action of a countable group $G$. Then $G\times X$ has the structure of a measured groupoid over $X$, where the source and range maps are given by $s(g,x)=x$ and $r(g,x)=gx$, the composition law is given by $(g,hx)(h,x)=(gh,x)$, the inverse map is given by $(g,x)^{-1}=(g^{-1},gx)$, and the units are $\mathsf{e}_x=(e,x)$. This measured groupoid is denoted by $G\ltimes X$.

Given a measured subgroupoid $\calh\subseteq\calg$, consider the equivalence relation $\sim_{\calh}$ on $\calg$ defined by $\mathsf{g}_1\sim_{\calh}\mathsf{g}_2$ if and only if $s(\mathsf{g}_1)=s(\mathsf{g}_2)$ and $\mathsf{g}_2\mathsf{g}_1^{-1}\in\calh$. We say that $\calh$ has \emph{finite index} in $\calg$ if for a.e.\ $x\in X$, there are only finitely many classes for $\sim_{\calh}$ on $s^{-1}(x)$.

A measured groupoid $\calg$ is of \emph{infinite type} if for every positive measure Borel subset $U\subseteq X$ and a.e.\ $x\in U$, there are infinitely many elements $\mathsf{g}\in\calg_{|U}$ with $s(\mathsf{g})=x$.

\paragraph{Cocycles.} Let $\calg$ be a measured groupoid over a standard probability space $X$, and $G$ be a countable group. A measurable map $\rho:\calg\to G$ is a \emph{strict cocycle} if for all $\mathsf{g},\mathsf{h}\in\calg$ with $s(\mathsf{g})=r(\mathsf{h})$, one has $\rho(\mathsf{gh})=\rho(\mathsf{g})\rho(\mathsf{h})$. 

A strict cocycle $\rho:\calg\to G$ \emph{has trivial kernel} if the only elements $\mathsf{g}\in\calg$ satisfying $\rho(\mathsf{g})=1$ are the units $\mathsf{e}_x$, with $x\in X$.

As an important example, if $\calg=G\ltimes X$ is the measured groupoid associated to a measure-preserving group action on a standard probability space $X$, then the map $(g,x)\mapsto g$ is a strict cocycle $\calg\to G$ with trivial kernel.

\paragraph*{Normalization.} Let $\calg$ be a measured groupoid over a standard probability space $X$, and let $\cala,\caln$ be two measured subgroupoids. 

Given a bisection $B\subseteq\calg$ with $U=s(B)$ and $V=r(B)$, we say that $\cala$ is \emph{$B$-invariant} if there exist conull Borel subsets $U^*\subseteq U$ and $V^*\subseteq V$ such that $B\cala_{|U^*}B^{-1}=\cala_{|V^*}$ (here $B\cala_{|U^*}B^{-1}$ is the set of all elements of the form $\mathsf{g}_1\mathsf{g}_2\mathsf{g}_3$ with $\mathsf{g}_1\in B$, $\mathsf{g}_2\in\cala_{|U^*}$ and $\mathsf{g}_3^{-1}\in B$). 

We say that $\cala$ is \emph{normalized} by $\caln$ if $\caln$ can be covered by countably many bisections $B_n$ (with $n\in\mathbb{N}$) in such a way that $\cala$ is $B_n$-invariant for every $n\in\mathbb{N}$.

We say that $\cala$ is \emph{stably normalized} by $\caln$ if there exists a partition $X^*=\dunion_{i\in I}X_i$ of a conull Borel subset $X^*\subseteq X$ into at most countably many Borel subsets, such that for every $i\in I$, the groupoid $\cala_{|X_i}$ is normalized by $\caln_{|X_i}$.

\begin{ex}\label{ex:normal}
Let $\calg$ be a measured groupoid over a standard probability space, let $G$ be a countable group, and let $\rho:\calg\to G$ be a strict cocycle. Let $A,N\subseteq G$ be subgroups such that $A$ is normalized by $N$. Then $\rho^{-1}(A)$ is normalized by $\rho^{-1}(N)$.
\end{ex}

\paragraph*{Amenability.} There is a notion of amenability of a measured groupoid which comes from the work of Zimmer \cite{Zim2}. When $\calg=G\ltimes X$ is the measured groupoid associated to an ergodic measure-preserving action of a countable group $G$ on a standard probability space $X$, the amenability of $\calg$ is equivalent to the amenability of the $G$-action on $X$ in the sense of Zimmer \cite[Definition~1.4]{Zim2}. Since we will never work with the actual definition of amenability of a measured groupoid (or a group action) in the present work, we omit it, and refer for instance to \cite[Definition~3.33]{GH}. All facts that we will need about amenable measured groupoids are collected in \cite[Section~3.3]{GH}. In particular, if $\rho:\calg\to G$ is a strict cocycle with trivial kernel towards a countable group $G$, and if $G$ is amenable, then $\calg$ is amenable, see e.g.\ \cite[Corollary~3.39]{GH}.
We say that a measured groupoid over a standard probability space $X$ is \emph{everywhere non-amenable} if for every positive measure Borel subset $U\subseteq X$, the groupoid $\calg_{|U}$ is non-amenable.

\paragraph*{Action-like cocycles.}

Following \cite[Definition~4.9]{EH}, a Borel map $\rho:\calg\to G$ is an \emph{action-like cocycle} if it is a strict cocycle and
\begin{enumerate}
\item $\rho$ has trivial kernel, i.e.\ $\rho(\mathsf{g})=1$ if and only if $\mathsf{g}=\mathsf{e}_x$ for some $x\in X$; 
\item whenever $H_1\subseteq H_2$ is an infinite-index inclusion of subgroups of $G$, for every positive measure Borel subset $U\subseteq X$, the measured subgroupoid $\rho^{-1}(H_1)_{|U}$ does not have finite index in $\rho^{-1}(H_2)_{|U}$;
\item for every non-amenable subgroup $H\subseteq G$, the measured subgroupoid $\rho^{-1}(H)$ is everywhere non-amenable.
\end{enumerate}

As a matter of fact \cite[Lemma~4.12]{EH}, which justifies the terminology, if $\calg=G\ltimes X$ is the measured groupoid associated to a measure-preserving action of a countable group $G$ on a standard probability space $X$, then the strict cocycle $(g,x)\mapsto g$ is action-like. We will also use the fact that if $\rho:\calg\to G$ is an action-like cocycle, and if $U\subseteq X$ is a positive measure Borel subset, then the restriction $\rho:\calg_{|U}\to G$ is again action-like.

\paragraph*{Cocycles towards right-angled Artin groups.}  This paragraph is specifically about the case where $\calg$ is equipped with a strict cocycle towards a right-angled Artin group $G$. Following \cite[Definition~3.6]{HH}, given a parabolic subgroup $P\subseteq G$, and a subgroupoid $\calp$ of $\calg$, we say that $(\calp,\rho)$ is \emph{tightly $P$-supported} if the following two hold:
\begin{itemize}
    \item there exists a conull Borel subset $X^*\subseteq X$ such that $\rho(\calp_{|X^*})\subseteq P$;
    \item for every positive measure Borel subset $U\subseteq X$, and every proper parabolic subgroup $Q\subsetneq P$, we have $\rho(\calp_{|U})\nsubseteq Q$. 
\end{itemize}
By \cite[Lemma~3.7]{HH},  there exists a partition $X=\sqcup_{i\in I} X_i$ into at most countably many Borel subsets
such that for every $i\in I$, there exists a parabolic subgroup $P_i$ such that $(\calp_{|X_i},\rho)$ is tightly $P_i$-supported. And by \cite[Lemma~3.8]{HH}, if $(\calp,\rho)$ is tightly $P$-supported, and if $\caln\subseteq\calg$ is a measured subgroupoid that normalizes $\calp$, then there exists a conull Borel subset $X^*\subseteq X$ such that $\rho(\caln_{|X^*})\subseteq P\times P^{\perp}$. 

\paragraph*{Relevance towards measure equivalence.}  Let $\Omega$ be a measure equivalence coupling between two countable groups $G$ and $H$, i.e.\ $\Omega$ is a standard measured space equipped with commuting measure-preserving free actions of $G$ and $H$ by Borel automorphisms, each having a Borel fundamental domain of finite positive measure. We can choose fundamental domains $X_G,X_H\subseteq\Omega$ for the respective actions of $G$ and $H$ on $\Omega$, such that the $(G\times H)$-translates of $U:=X_G\cap X_H$ cover $\Omega$ up to null sets (see \cite[Lemma~2.27]{Kid-survey}). The identification $X_G\simeq \Omega/G$ yields a measure-preserving $H$-action on $X_G$; likewise, we have a measure-preserving action of $G$ on $X_H$. The groupoids $G\ltimes X_H$ and $H\ltimes X_G$ have isomorphic restrictions to $U$; we denote by $\calg$ this restriction. Then $\calg$ is a measured groupoid over a standard finite measure space $U$ which comes with two action-like cocycles, one towards $G$ and one towards $H$.

\section{Non-abelian equivalence classes and measure equivalence}\label{sec:invariant}

The goal of the present section is to prove the following proposition. We refer to Definition~\ref{de:collapsible} for the notion of a clique-reduced right-angled Artin group; equivalence classes of vertices are understood in the sense of Charney--Vogtmann, as recalled on page~\pageref{equivalence-classes}.

\begin{prop}\label{prop:non-abelian-class}
Let $G$ and $H$ be two clique-reduced right-angled Artin groups that are measure equivalent. Then $\Gamma_G$ contains an untransvectable non-abelian equivalence class if and only if $\Gamma_H$ does.
\end{prop}

 Recall the notion of a collapsible parabolic subgroup from Definition~\ref{de:collapsible}. The following lemma extends Lemma~\ref{lemma:collapsible} to a groupoid-theoretic setting.

\begin{lemma}\label{lemma:coll}
Let $G$ be a right-angled Artin group, and let $\calg$ be a measured groupoid over a standard probability space $X$ equipped with an action-like cocycle $\rho:\calg\to G$. Let $Q\subseteq G$ be a parabolic subgroup, and let $\calq=\rho^{-1}(Q)$.

Then $Q$ is collapsible if and only if for every measured subgroupoid $\calb\subseteq\calq$ of infinite type, every measured subgroupoid of $\calg$ that stably normalizes $\calb$ also stably normalizes $\calq$.
\end{lemma}

\begin{proof}
We first assume that $Q$ is not collapsible, and aim to find a measured subgroupoid $\calb\subseteq\calq$ that violates the conclusion of the lemma. 

Since $Q$ is not collapsible, Lemma~\ref{lemma:collapsible} ensures that there exists an infinite parabolic subgroup $B\subseteq Q$ such that $B\times B^{\perp}\nsubseteq Q\times Q^{\perp}$. Let $\calb=\rho^{-1}(B)$, and let $\caln=\rho^{-1}(B\times B^{\perp})$. Then $\calb$ is of infinite type (because $B$ is infinite and $\rho$ is action-like), and $\caln$ normalizes $\calb$ (Example~\ref{ex:normal}). 

On the other hand we claim that $\caln$ does not stably normalize $\calq$. Indeed, assume towards a contradiction that it does. By \cite[Lemma~4.15]{EH}, the pair $(\calq,\rho)$ is tightly $Q$-supported. Since $\caln$ stably normalizes $\calq$, by \cite[Lemma~3.8]{HH}, there exists a positive measure Borel subset $U\subseteq X$ such that $\caln_{|U}\subseteq\rho^{-1}(Q\times Q^{\perp})_{|U}$. So $\rho(\caln_{|U})$ is contained in $(B\times B^{\perp})\cap (Q\times Q^\perp)$, which is a proper parabolic subgroup of $B\times B^{\perp}$. On the other hand, since $\caln=\rho^{-1}(B\times B^{\perp})$ and $\rho$ is action-like, \cite[Lemma~4.15]{EH} ensures that $(\caln,\rho)$ is tightly $(B\times B^{\perp})$-supported. This is a contradiction, which proves our claim. 

Therefore the subgroupoid $\calb$ violates the conclusion of the lemma, as desired.

\medskip

We now assume that $Q$ is collapsible. Let $\calb\subseteq\calq$ be a measured subgroupoid of infinite type. Up to replacing $X$ by a conull Borel subset and taking a countable Borel partition, \cite[Lemma~3.7]{HH} allows us to assume that $(\calb,\rho)$ is tightly $B$-supported for some parabolic subgroup $B$ of $G$. Since $\calb\subseteq\calq$, we have $B\subseteq Q$. Since $\calb$ is of infinite type and $\rho$ has trivial kernel, the group $B$ is infinite. Since $Q$ is collapsible, Lemma~\ref{lemma:collapsible} ensures that $B\times B^{\perp}\subseteq Q\times Q^{\perp}$. Let $\caln$ be a measured subgroupoid of $\calg$ which stably normalizes $\calb$. Up to a conull Borel subset and a countable Borel partition of the base space, we will assume that $\caln$ normalizes $\calb$. By \cite[Lemma~3.8]{HH}, up to replacing $X$ by a conull Borel subset, we have $\caln\subseteq\rho^{-1}(B\times B^{\perp})$. In particular $\caln\subseteq\rho^{-1}(Q\times Q^{\perp})$. Since $\rho^{-1}(Q\times Q^{\perp})$ normalizes $\calq$ (Example~\ref{ex:normal}), we deduce that $\caln$ normalizes $\calq$, as desired.
\end{proof}

The following lemma is the groupoid-theoretic version of the following observation: a right-angled Artin group is a free product of free abelian groups if and only if it does not contain any subgroup isomorphic to $\mathbb{Z}\times F_2$.

\begin{lemma}\label{lemma:thin}
Let $G$ be a right-angled Artin group, and let $\calg$ be a measured groupoid over a standard probability space $X$, equipped with an action-like cocycle $\rho:\calg\to G$.

Then $G$ is a free product of free abelian groups if and only if for every positive measure Borel subset $U\subseteq X$, there do not exist measured subgroupoids $\cala,\caln\subseteq\calg_{|U}$ such that $\cala$ is amenable and of infinite type, and $\caln$ is everywhere non-amenable and normalizes $\cala$.
\end{lemma}

\begin{proof}
We first assume that $G$ is not a free product of free abelian groups. Then $G$ contains a subgroup $N$ isomorphic to $\mathbb{Z}\times F_2$. We let $A$ be the center of $N$. We let $\cala=\rho^{-1}(A)$ and $\caln=\rho^{-1}(N)$. Then $\cala$ is amenable (using that $\rho$ has trivial kernel, see e.g.\ \cite[Corollary~3.39]{GH}) and of infinite type (using that $A$ is infinite and $\rho$ is action-like). And $\caln$ is everywhere non-amenable (because $N$ is non-amenable and $\rho$ is action-like), and $\caln$ normalizes $\cala$ (Example~\ref{ex:normal}). This shows that the property from the lemma fails.

Conversely, let us assume that $\calg$ fails to satisfy the property of the lemma, and let us prove that $G$ contains a subgroup isomorphic to $\mathbb{Z}\times F_2$. By assumption, there exists a positive measure Borel subset $U\subseteq X$ and measured groupoids $\cala,\caln\subseteq\calg_{|U}$ as in the lemma. By \cite[Lemma~3.7]{HH}, we can find a positive measure Borel subset $V\subseteq U$ such that $(\cala_{|V},\rho)$ is tightly $A$-supported for some parabolic subgroup $A$ of $G$. The group $A$ is infinite because $\cala$ is of infinite type and $\rho$ has trivial kernel. Since $\caln_{|V}$ is non-amenable and normalizes $\cala_{|V}$, it follows from \cite[Lemma~3.10]{HH} that $A^{\perp}$ is non-amenable, and therefore $G$ contains a parabolic subgroup isomorphic to $\mathbb{Z}\times F_2$, as desired.
\end{proof}

We are now in position to complete the proof of the main proposition of the section.

\begin{proof}[Proof of Proposition~\ref{prop:non-abelian-class}]
As explained in the last paragraph of Section~\ref{sec:groupoids}, we can find a measured groupoid $\calg$ over a standard probability space $X$, equipped with two action-like cocycles $\rho_G:\calg\to G$ and $\rho_H:\calg\to H$. Assuming that $\Gamma_G$ contains a non-abelian untransvectable equivalence class, we will prove that so does $\Gamma_H$.

Let $\Theta\subseteq\Gamma_G$ be a subgraph associated with an untransvectable non-abelian equivalence class of $G$, let $Q=G_\Theta$, and let $\calq=\rho_G^{-1}(Q)$. By induction, we construct a sequence of full subgraphs
 \[\Theta\subseteq\Upsilon_n\subsetneq\Lambda_n\subsetneq\Upsilon_{n-1}\subsetneq\dots\subsetneq\Upsilon_1\subsetneq\Lambda_1\subseteq\Upsilon_0=\Gamma_G,\] coming with associated parabolic subgroups $P_i=G_{\Lambda_i}$ and $F_i=G_{\Upsilon_i}$, in the following way: we let $F_0=G$, and for every $i\in\{1,\dots,n\}$, 
\begin{itemize}
\item $\Lambda_i$ is the type of a maximal product parabolic subgroup $P_i$ of $F_{i-1}$ containing $(Q\times Q^{\perp})\cap F_{i-1}$ (if this exists);
\item $\Upsilon_i$ is the type of the unique irreducible factor $F_i$ of $P_i$ containing $Q$ (notice indeed that since $Q$ is a free parabolic subgroup, it must be entirely contained in a single irreducible factor of $P$).
\end{itemize}
We stop when $(Q\times Q^{\perp})\cap F_i$ is not contained in any product parabolic subgroup of $F_i$. Note that all inclusions in the above sequence (except possibly $\Lambda_1\subseteq\Upsilon_0$) are indeed strict: this is because $F_i$ is an irreducible factor of the product parabolic subgroup $P_i$ (showing that $\Upsilon_i\subsetneq \Lambda_i$), and $P_i$ is a product parabolic subgroup in $F_{i-1}$, which is irreducible for $i\ge 1$ (showing that $\Lambda_i\subsetneq\Upsilon_{i-1}$). In particular, the process terminates after finitely many steps.

For every $i\in\{1,\dots,n\}$, we let $\bar F_i:=F_i^{\perp}\cap F_{i-1}$.  Since $F_i$ is a direct factor of $P_i$, and $P_i\subseteq F_{i-1}$, we have $P_i\subseteq F_i\times \bar F_i$. Then the maximality of $P_i$ (as a product parabolic subgroup of $F_{i-1}$) ensures that $P_i=F_i\times \bar F_i$, and $\bar F_i$ is the product of all irreducible factors of $P_i$ except $F_i$ (in particular $\bar F_i\neq\{1\}$). Since $Q\subseteq F_i$,  we have $\bar F_i\subset Q^\perp$. Since in addition $Q^{\perp}\cap F_{i-1}\subseteq P_i$ by construction, we have 
\begin{equation}\label{eq:qperp}
Q^\perp\cap F_{i-1}=(Q^\perp\cap F_i)\times \bar F_i. 
\end{equation}

\begin{claim}
For every $i\in\{0,\dots,n\}$, the subgroup $Q$ is untransvectable inside $F_i$, i.e.\ it is untransvectable when viewed as a parabolic subgroup of $F_i$.
\end{claim}

\begin{proof}[Proof of the claim]
We argue by induction. The case $i=0$ holds by our choice of $Q$. So let $i\in\{1,\dots,n\}$; assuming that the property holds for $i-1$, let us prove it for $i$.

Assume towards a contradiction that $Q$ fails to be untransvectable inside $F_{i}$. Then there exists a standard cyclic parabolic subgroup $Z\subseteq F_{i}$, not contained in $Q$, such that $Q^{\perp}\cap F_{i}\subseteq Z\times Z^{\perp}$. Since $Z\subseteq F_i$, we have $\bar F_{i}\subseteq Z^{\perp}$, so Equation~\eqref{eq:qperp} implies that $Q^\perp\cap F_{i-1}\subseteq Z\times Z^{\perp}$. Therefore $Q$ is transvectable inside $F_{i-1}$, contradicting our induction hypothesis.  
\end{proof}

\begin{claim}\label{claim:trivial-center}
For every $i\in\{1,\dots,n\}$, the subgroup $P_i$ has trivial center.
\end{claim}

\begin{proof}[Proof of the claim]
Otherwise $(Q\times Q^{\perp})\cap F_{i-1}$ is contained in $Z\times Z^{\perp}$ for some standard cyclic parabolic subgroup $Z$ contained in the center of $P_i$. Notice that $Z$ is not contained in $Q$ because $Q$ is a non-abelian free group. This contradicts the previous claim.
\end{proof}

\begin{claim}\label{claim:fnb}
One has $F_n=Q$.
\end{claim}

\begin{proof}[Proof of the claim]
Since $(Q\times Q^{\perp})\cap F_{n}$ is not a product parabolic subgroup, we must have $Q^\perp\cap F_n=\{1\}$. It follows from Equation~\eqref{eq:qperp} that $Q^{\perp}\cap F_{n-1}=\bar F_n$ (this still holds in the degenerate case where $n=0$, with the conventions $F_{-1}=G$ and $\bar F_0=\{1\}$). If there exists a cyclic parabolic subgroup $Z$ of $F_n$ not in $Q$, then $Q^{\perp}\cap F_{n-1}\subseteq Z^\perp$, contradicting that $Q$ is untransvectable inside $F_{n-1}$.
\end{proof}

For every $i\in\{1,\dots,n\}$, let $\calp_i=\rho_G^{-1}(P_i)$ and $\calf_i=\rho_G^{-1}(F_i)$. In particular $\calq=\calf_n$.

The group $P_1$ is a maximal product parabolic subgroup of $F_0$, and is non-abelian (Claim~\ref{claim:trivial-center}). Since $\rho_G$ is action-like, it follows that $\calp_1$ is everywhere non-amenable, and that for every positive measure Borel subset $U\subseteq X$, the pair $((\calp_1)_{|U},\rho_G)$ is tightly $P_1$-supported \cite[Lemma~4.15]{EH}. In particular $(\calp_1,\rho_G)$ is \emph{nowhere of isolated clique type} in the sense of \cite[Section~7.2]{EH}. We can therefore apply \cite[Lemma~7.4(1)]{EH} and deduce that $(\calg,\calp_1)$ satisfies Property~$(\mathrm{P}_{\mathrm{prod}})$ from \cite[Definition~7.2]{EH}. We now apply \cite[Lemma~7.4(2)]{EH} to the cocycle $\rho_H$. It implies that there exists a positive measure Borel subset $U\subseteq X$ such that either  $(\calp_1)_{|U}=\rho_H^{-1}(P_1^H)_{|U}$ for some maximal product parabolic subgroup $P_1^H$ of $H$, or else $(\calp_1)_{|U}$ is contained in a subgroupoid $\calq$ of $\calg_{|U}$ such that $(\calq,\rho_H)$ is of isolated clique type. But the second case is excluded as it would imply that $(\calp_1)_{|U}$ is amenable. So $(\calp_1)_{|U}=\rho_H^{-1}(P_1^H)_{|U}$ for some maximal product parabolic subgroup $P_1^H$ of $H$, and using again that $\calp_1$ is everywhere non-amenable, we deduce that $P_1^H$ is non-abelian. 

Since $P_1$ has trivial center (Claim~\ref{claim:trivial-center}), it follows from \cite[Lemma~9.6]{EH} that $P_1^H$ has trivial center (indeed, in the terminology from \cite{EH}, this is characterized by the fact that there is no subgroupoid $\calc$ of $(\calp_1)_{|U}$ such that $((\calp_1)_{|U},\calc)$ satisfies Property~$(\mathrm{P}_{\mathrm{clique}})$, a property that is phrased with no reference to any action-like cocycle).

Finally, since $F_1$ is an irreducible factor of $P_1$, it follows from \cite[Lemma~8.5]{EH} that we can find a positive measure Borel subset $V\subseteq U$ such that $(\calf_1)_{|V}=\rho_H^{-1}(F_1^H)_{|V}$ for some irreducible factor $F_1^H$ of $P_1^H$. 

 Repeating this procedure inductively on $i$, we obtain a sequence of parabolic subgroups $F^H_i$ and $P^H_i$ of $H$, and a positive measure Borel subset $W\subseteq X$, such that $(\calf_i)_{|W}=\rho_H^{-1}(F_i^H)_{|W}$ and $(\calp_i)_{|W}=\rho_H^{-1}(P_i^H)_{|W}$, where $F^H_i$ is an irreducible factor of $P^H_i$, and $P^H_i$ is a maximal product parabolic subgroup in $F^H_{i-1}$ with trivial center.

\medskip

In particular $\calq_{|W}=\rho_H^{-1}(F_n^H)_{|W}$. It remains to show that the parabolic subgroup $F_n^H$ is collapsible, non-abelian free, and untransvectable, as this is enough to conclude that the type of $F_n^H$ is a non-abelian equivalence class (see Remark~\ref{rk:na}). This is done in the three points below.

$\bullet$ Since $F_n=Q$ (Claim~\ref{claim:fnb}), it is collapsible. Since $\rho_G^{-1}(F_n)_{|W}=\rho_H^{-1}(F_n^H)_{|W}$, we can apply Lemma~\ref{lemma:coll} in the ambient groupoid $\calg_{|W}$: this lemma gives a characterization of collapsibility that is independent of any action-like cocycle, so we deduce that $F_n^H$ is also collapsible. 

$\bullet$ Since $F_n=Q$ is a free product of free abelian groups, Lemma~\ref{lemma:thin} (applied within the groupoid $\calg_{|W}$) ensures that $F_n^H$ is a free product of free abelian groups (in fact $F_n^H$ is a free group, using that $H$ is clique-reduced and $F^H_n$ is collapsible). And since $F_n=Q$ is non-amenable and $\rho_G$ is action-like, the groupoid $\calq=\calf_n$ is everywhere non-amenable. Using that $\rho_H$ has trivial kernel, it follows that $F_n^H$ is also non-abelian (by e.g.\ \cite[Corollary~3.39]{GH}).

$\bullet$ We finally prove that $F_n^H$ is untransvectable. For this, after collapsing the subgraph corresponding to the type of $F_n^H$, we view $H$ as a graph product where one vertex group (associated with a vertex which we denote $v_0$) is a non-abelian free group, and all other vertex groups are isomorphic to $\mathbb{Z}$. We have just constructed a sequence \[F_n^H\subseteq P_n^H\subseteq\dots\subseteq P_1^H\subseteq F_0^H=H\] of parabolic subgroups of $H$ -- and they are still parabolic subgroups for the new graph product structure because they all contain $F_n^H$. Each $P_i^H$ is a maximal product parabolic subgroup of $F_i^H$, and the type of $P_i^H$ (in the new graph product structure) is not a clique because $P_i^H$ has trivial center. Also, since $P_i^H$ has trivial center, its type cannot have a vertex $v$ joined to all other vertices, except if this vertex is $v_0$, which only happens when $i=n$. Therefore \cite[Lemma~6.6]{EH} applies and shows that $F_n^H$ is untransvectable.
\end{proof}

\section{Strongly untransvectable cyclic parabolic subgroups}

Let $G$ be a right-angled Artin group. We say that an infinite cyclic parabolic subgroup $P$ is \emph{strongly untransvectable} if it is untransvectable, and the common normalizer of all untransvectable cyclic parabolic subgroups that commute with $P$ is equal to $P$. We will likewise say that a vertex $v\in V\Gamma_G$ is \emph{strongly untransvectable} if the parabolic subgroup $G_v$ is strongly untransvectable.

\begin{rk}\label{rk:strongly-untransvectable}
When $G$ is transvection-free, every cyclic parabolic subgroup $P$ is strongly untransvectable, as a consequence of the fact that $(\{v\}^{\perp})^{\perp}=\{v\}$, for $v\in V\Gamma_G$ giving the type of $P$.  
\end{rk}

\begin{prop}\label{prop:strongly-untransvectable}
Let $G$ and $H$ be two clique-reduced right-angled Artin groups that are measure equivalent. If every untransvectable cyclic parabolic subgroup of $G$ is strongly untransvectable, then the same holds in $H$. 
\end{prop}

\begin{proof}
As explained in the last paragraph of Section~\ref{sec:groupoids}, we can find a measured groupoid $\calg$ over a standard probability space $X$, equipped with two action-like cocycles $\rho_G:\calg\to G$ and $\rho_H:\calg\to H$.

Let $B\subseteq H$ be an untransvectable cyclic parabolic subgroup. Let $\cala=\rho_H^{-1}(B)$. By \cite[Proposition~9.1]{EH}, there exist an untransvectable cyclic parabolic subgroup $A\subseteq G$ and a positive measure Borel subset $U\subseteq X$ such that $\cala_{|U}=\rho_G^{-1}(A)_{|U}$. Since $A$ is strongly untransvectable, and since every non-ascending chain of parabolic subgroups terminates, we can find a finite collection $A_1,\dots,A_k$ of untransvectable cyclic parabolic subgroups commuting with $A$, such that $N_G(A_1)\cap\dots\cap N_G(A_k)=A$. For every $i\in\{1,\dots,k\}$, let $\cala_i=\rho_G^{-1}(A_i)$, and let $\caln_i=\rho_G^{-1}(N_G(A_i))$. Applying \cite[Proposition~9.1]{EH} again, there exist a positive measure Borel subset $V\subseteq U$, and untransvectable cyclic parabolic subgroups $B_1,\dots,B_k$, such that for every $i\in\{1,\dots,k\}$, one has $(\cala_i)_{|V}=\rho_H^{-1}(B_i)_{|V}$.

Since all subgroups $A_i$ commute with $A$, the subgroupoids $(\cala_i)_{|V}$ all normalize $\cala_{|V}$ (Example~\ref{ex:normal}, applied to the cocycle $\rho_G$). Since $(\cala_{|V},\rho_H)$ is tightly $B$-supported and $((\cala_i)_{|V},\rho_H)$ is tightly $B_i$-supported \cite[Lemma~4.15]{EH}, it follows from \cite[Lemma~3.8]{HH} that $B_i\subseteq B\times B^{\perp}$, so $B$ and $B_i$ commute. 

The subgroupoid $\caln_i$ normalizes $\cala_i$, and $((\cala_i)_{|V},\rho_H)$ is tightly $B_i$-supported. It thus follows from \cite[Lemma~3.8]{HH} (applied to the cocycle $\rho_H$) that there exists a conull Borel subset $V^*\subseteq V$ such that $(\caln_i)_{|V^*}\subseteq\rho_H^{-1}(N_H(B_i))_{|V^*}$. On the other hand, \cite[Lemma~3.8]{HH} (applied to the cocycle $\rho_G$) shows that for every measured subgroupoid $\calh$ of $\calg_{|V}$ that normalizes $(\cala_i)_{|V}$,  there exists a conull Borel subset $V^{**}\subseteq V$ such that $\calh_{|V^{**}}\subseteq (\caln_i)_{|V^{**}}$. Since $\rho_H^{-1}(N_H(B_i))_{|V}$ normalizes $(\cala_i)_{|V}=\rho_H^{-1}(B_i)_{|V}$ (Example~\ref{ex:normal}), it follows that, up to replacing $V^*$ by a further conull Borel subset of $V$, the reverse inclusion  $\rho_H^{-1}(N_H(B_i))_{|V^*}\subseteq (\caln_i)_{|V^*}$ also holds, so $(\caln_i)_{|V^*}=\rho_H^{-1}(N_H(B_i))_{|V^*}$.

We have $(\caln_1\cap\dots\cap\caln_k)_{|V^*}=\rho_G^{-1}(N_G(A_1)\cap\dots\cap N_G(A_k))_{|V^*}=\cala_{|V^*}$, which means that $\rho_H^{-1}(N_H(B_1)\cap\dots\cap N_H(B_k))_{|V^*}=\rho_H^{-1}(B)_{|V^*}$. Since $\rho_H$ is action-like, it follows that the inclusion $B\subseteq N_H(B_1)\cap\dots\cap N_H(B_k)$ has finite index. Since this is an inclusion of parabolic subgroups, we must have $B=N_H(B_1)\cap\dots\cap N_H(B_k)$, which shows that $B$ is strongly untransvectable.
\end{proof}

\section{A combinatorial rigidity statement}\label{sec:rigid}

Recall that the \emph{extension graph} of a right-angled Artin group $G$, denoted $\Gamma_G^e$, is the simple graph whose vertices are the parabolic subgroups of $G$ isomorphic to $\mathbb{Z}$, where two distinct vertices are adjacent if the corresponding parabolic subgroups commute \cite{KK}. The $G$-action by conjugation on its cyclic parabolic subgroups induces a $G$-action by graph automorphisms on $\Gamma_G^e$. Following \cite[Definition~9.3]{EH}, we define the \emph{untransvectable extension graph} $\Gamma_G^{ue}$ as the full subgraph of $\Gamma_G^e$ spanned by all untransvectable cyclic parabolic subgroups (notice that it is $G$-invariant).

\begin{prop}\label{prop:rigidity}
Let $G$ and $H$ be two clique-reduced right-angled Artin groups. Assume that $\Out(G)$ is finite, and $G$ is not cyclic. Also assume that 
\begin{itemize}
\item $\Gamma_H$ does not contain any non-abelian untransvectable equivalence class;
\item every untransvectable vertex of $\Gamma_H$ is strongly untransvectable. 
\end{itemize}
Then  $\Gamma_H^{ue}$ is isomorphic to $\Gamma_G^{ue}$ if and only if $H$ is isomorphic to a finite-index subgroup of $G$. 
\end{prop}

\begin{ex}\label{ex:counterexample}
We do not know whether the first assumption on $\Gamma_H$ can be dropped. But the second assumption cannot, as shown by the following example.

Let $G$ be a non-cyclic right-angled Artin group with $|\Out(G)|<+\infty$. Choose a vertex $v\in V\Gamma_G$ (notice that our assumptions on $G$ imply that $v$ is not isolated in $\Gamma_G$). Let $H_v$ be a right-angled Artin group that splits as $H_v=\mathbb{Z}\times H'_v$, where $H'_v$ is a right-angled Artin group with trivial center. Let $H$ be the graph product over $\Gamma_G$ where the vertex group associated with $v$ is $H_v$, and all other vertex groups are $\mathbb{Z}$. Then $\Gamma_H$ is obtained from $\Gamma_G$ by replacing the vertex $v$ by $\Gamma_{H_v}$, and joining all vertices of $\Gamma_{H_v}$ to all vertices of $\Gamma_G$ that are adjacent to $v$. Let $v_0\in V\Gamma_{H_v}$ be the vertex associated to the center of $H_v$; we also view $v_0$ as a vertex in $\Gamma_H$. Notice that $v_0$ is not strongly untransvectable, as $(\{v_0\}^{\perp})^{\perp}$ contains all vertices of $\Gamma_G$ adjacent to $v$.

The untransvectable vertices in $\Gamma_H$ are exactly $v_0$ and all vertices coming from $\Gamma_G$ and distinct from $v$.

Let $\theta_v:\mathbb{Z}\to H_v$ be a bijection sending the neutral element $0$ to the neutral element of $H_v$. Using normal forms for graph products as in \cite[Theorem~3.9]{Gre}, the bijection $\theta_v$ extends to a bijection $\theta:G\to H$.

We claim that the map $\theta_\ast:\Gamma_G^e\to\Gamma_H^{ue}$ sending $gG_wg^{-1}$ to $\theta(g)H_w\theta(g)^{-1}$ if $w\neq v$, and $gG_vg^{-1}$ to $\theta(g)H_{v_0}\theta(g)^{-1}$, is a graph isomorphism. Notice that this is well-defined, because if $h$ normalizes $G_w$ (resp.\ $G_v$), then $\theta(h)$ normalizes $H_w$ (resp.\ $H_{v_0}$) -- this uses that $H_{v_0}$ is central in $H_v$. It is bijective, using $\theta^{-1}$ to define the inverse, and it preserves adjacency and non-adjacency. 
\end{ex}

 In our proof of Proposition~\ref{prop:rigidity}, we will make use of the following lemma of the second-named author.

\begin{lemma}[{\cite[Lemma~6.2]{Hua} and \cite[Corollary~3.17(1)]{Hua2}}]\label{lemma:projection}
Let $G$ be a non-cyclic right-angled Artin group. Let $v\in V\Gamma_G^e$, and let $Z_v$ be the cyclic parabolic subgroup of $G$ associated with $v$. Then $Z_v$ preserves $V\Gamma_G^e\setminus\st_{\Gamma^e_G}(v)$, and there is a map $\theta:V\Gamma_G^e\setminus\st_{\Gamma_G^e}(v)\to Z_v$ such that
\begin{itemize}
    \item[(1)] $\theta$ is $Z_v$-equivariant (for the action of $Z_v$ by left translations on itself);
    \item[(2)] $\theta$ is constant on each connected component of $\Gamma_G^e\setminus\st_{\Gamma_G^e}(v)$;
    \item[(3)] if $|\Out(G)|<+\infty$, then for every $x\in Z_v$, the full subgraph of $\Gamma_G^e$ spanned by the vertices in $\theta^{-1}(x)$ is connected.
\end{itemize}
\end{lemma}

\begin{rk}
A map $\theta$ satisfying (1) and (2) is constructed in \cite[Lemma~6.2]{Hua} (its equivariance is not stated explicitly, but follows from its construction). When $|\Out(G)|<+\infty$, the fact that it satisfies (3) is a consequence of \cite[Corollary~3.17(1)]{Hua2}. More precisely, in the terminology from \cite[Corollary~3.17]{Hua2}, $v$-branches are the connected components of the full subgraph of $\Gamma_G^e$ spanned by the vertices in $V\Gamma_G^e\setminus\st_{\Gamma_G^e}(v)$, and a $v$-tier is the full subgraph spanned by the vertices in some preimage of $\theta$ (see \cite[Definition~3.2]{Hua2}). In particular, a $v$-tier is a disjoint union of $v$-branches. Since $|\Out(G)|<+\infty$, for each vertex  $\bar v\in V\Gamma_G$, the complement $\Gamma_G\setminus \st_{\Gamma_G}(\bar v)$ is connected, and each vertex of $\lk_{\Gamma_G}(\bar v)$ is adjacent to a vertex in  $\Gamma_G\setminus \st_{\Gamma_G}(\bar v)$. Using the notations after \cite[Lemma~3.12]{Hua2}, this means that there is a single connected component $C$ of $\Gamma_G\setminus\st_{\Gamma_G}(\bar v)$, and that $\partial C=\lk_{\Gamma_G}(\bar v)$. Sticking to the terminology from \cite{Hua2}, the universal cover of the Salvetti complex of $G$ has a subcomplex $P_v$ isomorphic to $\mathbb R\times Q_v$, where $Q_v$ is isomorphic to the universal cover of the Salvetti complex of $G_{\lk_{\Gamma_G}(\bar v)}$. A \emph{$v$-peripheral subcomplex} of type $\partial C$ is then a subcomplex $K$ of $P_v$ of the form $\{x\}\times Q_v$ with $x\in\mathbb{Z}$, and there is only one of any given height $x$. In particular, for a given height $x$, there is a unique pair $(C,K)$ as in \cite[Corollary~3.17(1)]{Hua2}. It thus follows from \cite[Corollary~3.17(1)]{Hua2} that each $v$-tier (of a given height) contains exactly one $v$-branch. In other words $\theta^{-1}(x)$ (which is a $v$-tier) coincides with a single connected component (a $v$-branch) of the full subgraph of $\Gamma_G^e$ spanned by the vertices in $V\Gamma_G^e\setminus\st_{\Gamma_G^e}(v)$. This exactly means that (3) holds.
\end{rk}

\begin{lemma}\label{lemma:disconnected}
Let $G$ be a right-angled Artin group. Let $v\in V\Gamma_G^e$, and let $\Upsilon\subseteq\Gamma_G^e$ be a $G$-invariant full subgraph which is not contained in $\st_{\Gamma_G^e}(v)$.

Then $\Upsilon\setminus (\st_{\Gamma^e_G}(v)\cap\Upsilon)$ is disconnected.
\end{lemma}

\begin{proof}
The lemma is vacuously true if $G\simeq\mathbb{Z}$, so we will assume otherwise. Let $Z_v$ be the cyclic parabolic subgroup of $G$ associated with $v$, and let $g_v$ be a generator of $Z_v$.  By Lemma~\ref{lemma:projection}, there exists a $Z_v$-equivariant map $\theta:V\Gamma_G^e\setminus\st_{\Gamma_G^e}(v)\to Z_v$ (for the action of $Z_v$ by translations on itself) which is constant on each connected component of $\Gamma_G^e\setminus\st_{\Gamma_G^e}(v)$. Let $w\in V\Upsilon\setminus (\st_{\Gamma_G^e}(v)\cap\Upsilon)$. Then $w$ and $g_v\cdot w$ have distinct values under $\theta$, so they belong to different complementary components of $\st_{\Gamma_G^e}(v)$. In particular, in $\Upsilon$, they belong to distinct connected components of $\Upsilon\setminus (\st_{\Gamma_G^e}(v)\cap\Upsilon)$, which proves the lemma. 
\end{proof}

\begin{lemma}\label{lemma:proper-in-star}
Let $G$ be a right-angled Artin group with $|\Out(G)|<+\infty$. Then for every vertex $v\in V\Gamma_G^e$ and every full subgraph $X\subsetneq\st_{\Gamma_G^e}(v)$ with $X\neq\lk_{\Gamma_G^e}(v)$, the graph $\Gamma_G^e\setminus X$ is connected.
\end{lemma}

\begin{proof}
The lemma is vacuously true if $G\simeq\mathbb{Z}$, so we will assume otherwise. Let $Z_v\subseteq G$ be the cyclic parabolic subgroup associated to $v$, and let $g_v$ be a generator of $Z_v$. Since $|\Out(G)|<+\infty$, Lemma~\ref{lemma:projection} provides us with a map $\theta:V\Gamma_G^e\setminus\st_{\Gamma_G^e}(v)\to Z_v$ such that
\begin{itemize}
    \item[(1)] $\theta$ is $Z_v$-equivariant (for the action of $Z_v$ on itself by left translations);
    \item[(2)] for every $x\in Z_v$, the full subgraph of $\Gamma_G^e$ spanned by the vertices in $\theta^{-1}(x)$ is connected.
\end{itemize}
Our assumptions on $X$ ensure that we can find a vertex  $w\in\lk_{\Gamma_G^e}(v)\setminus X$. Let $C$ be the connected component of $\Gamma_G^e\setminus X$ that contains $w$. Let $\bar v, \bar w$ be the projections of $v,w$ to $\Gamma_G$, i.e.\ the types of the corresponding cyclic parabolic subgroups. Since $|\Out(G)|<+\infty$, the link of $\bar w$ in $\Gamma_G$ contains a vertex which is not adjacent to $\bar v$. Therefore in $\Gamma_G^e$, we can find a vertex $u$ adjacent to $w$ and not contained in $\st_{\Gamma_G^e}(v)$. As $n$ varies in $\mathbb{Z}$, the vertices $g_v^n\cdot u$ (corresponding to the conjugates $g_v^nZ_ug_v^{-n}$) are all adjacent to $w$ and not in $\st_{\Gamma_G^e}(v)$, in particular they belong to $C$. By (1) above, the values of $\theta(g_v^n\cdot u)$ as $n$ varies in $\mathbb{Z}$ cover all $Z_v$, so by (2) every element of $V\Gamma_G^e\setminus\st_{\Gamma_G^e}(v)$ belongs to $C$. Since every element of $\st_{\Gamma_G^e}(v)\setminus X$ is in the same connected component as some vertex outside $\st_{\Gamma_G^e}(v)$, we deduce that $\Gamma_G^e\setminus X=C$. 
\end{proof}

\begin{lemma}\label{lemma:combinatorial}
Let $H$ be a clique-reduced right-angled Artin group, and assume that $\Gamma_H$ contains a transvectable vertex. Then one of the following two holds:
\begin{itemize}
    \item[(1)] $\Gamma_H$ contains an untransvectable non-abelian equivalence class of vertices;
    \item[(2)] $\Gamma_H$ contains two vertices $u,\hat u$, with $\hat u$ untransvectable, such that $\lk_{\Gamma_H}(u)\subseteq\st_{\Gamma_H}(\hat u)$.
\end{itemize}
\end{lemma}

\begin{proof}
Let $u\in V\Gamma_H$ be a transvectable vertex such that there is no transvectable vertex $u'$ of $\Gamma_H$ such that $u<u'$ (the relation $<$ is defined on page~\pageref{equivalence-classes}).

We first assume that there exists an untransvectable vertex $\hat u$ with $u\le \hat u$. This means that $\lk_{\Gamma_H}(u)\subseteq\st_{\Gamma_H}(\hat u)$, showing that (2) holds.

We now assume that there is no untransvectable vertex $\hat u$ with $u\le \hat u$. Since $u$ is transvectable, there exists a vertex $w$ with $u\le w$, that is to say such that $\lk_{\Gamma_H}(u)\subseteq \st_{\Gamma_H}(w)$, so necessarily $w$ is transvectable. The maximality property defining $u$ thus ensures that $u\sim w$. Since $H$ is clique-reduced, the equivalence class $[u]$ of $u$ is non-abelian. And it is untransvectable, because there is no vertex $w'$ of $\Gamma_H$ (either transvectable or untransvectable) with $u<w'$. This shows that (1) holds.
\end{proof}

We are now in position to prove the main proposition of the section.

\begin{proof}[Proof of Proposition~\ref{prop:rigidity}]
We first assume that $H$ is isomorphic to a finite-index subgroup of $G$. By \cite[Theorem~5.3]{Hua}, the group $H$ is also transvection-free, so $\Gamma_H^{ue}=\Gamma_H^e$. The fact that $\Gamma_G^e$ and $\Gamma_H^{ue}$ are isomorphic thus follows from \cite[Theorem~1.2]{Hua}.

Conversely, we now assume that $\Gamma_H^{ue}$ is isomorphic to $\Gamma_G^e$, and aim to prove that $H$ embeds as a finite-index subgroup of $G$. If $H$ is transvection-free, then this follows from \cite[Theorem~1.2]{Hua}. 

We now assume that $H$ is not transvection-free (i.e.\ $\Gamma_H$ contains a transvectable vertex), and we aim for a contradiction. Since by assumption $\Gamma_H$ does not contain any untransvectable non-abelian equivalence class,  Lemma~\ref{lemma:combinatorial} ensures that $\Gamma_H$ contains a transvectable vertex $u$ and an untransvectable vertex $\hat{u}$ with $\lk_{\Gamma_H}(u)\subseteq\st_{\Gamma_H}(\hat{u})$. 

We now view $u$ and $\hat{u}$ as vertices of $\Gamma_H^e$, by identifying these vertices with the corresponding cyclic parabolic subgroups $H_u$ and $H_{\hat{u}}$. 

\begin{claim}\label{claim:proper}
The following two hold:
\begin{itemize}
\item $\lk_{\Gamma_H^e}(u)\cap\Gamma_H^{ue}\subsetneq\st_{\Gamma_H^e}(\hat u)\cap\Gamma_H^{ue}$, 
\item $\lk_{\Gamma_H^e}(u)\cap\Gamma_H^{ue}\neq\lk_{\Gamma_H^e}(\hat u)\cap\Gamma_H^{ue}$. 
\end{itemize}
\end{claim}

Before proving Claim~\ref{claim:proper}, let us first explain how the proposition follows from this claim.  Since $G$ is not cyclic, $\Gamma_G^e$ is not contained in the star of one vertex, so neither is $\Gamma_H^{ue}$. We can therefore apply Lemma~\ref{lemma:disconnected} to the right-angled Artin group $H$, to the vertex $u$, with $\Upsilon=\Gamma_H^{ue}$. This shows that  $\st_{\Gamma_H^e}(u)\cap\Gamma_H^{ue}$ disconnects $\Gamma_H^{ue}$. Through the isomorphism $\Gamma_H^{ue}\simeq\Gamma_G^e$, we now think of $\st_{\Gamma_H^e}(u)\cap\Gamma_H^{ue}$ as disconnecting $\Gamma_G^e$. On the other hand Claim~\ref{claim:proper} ensures that $\st_{\Gamma_H^e}(u)\cap\Gamma_H^{ue}$ (which is equal to $\lk_{\Gamma_H^e}(u)\cap\Gamma_H^{ue}$ because $u\notin V\Gamma_H^{ue}$) is a proper subgraph of $\st_{\Gamma_G^e}(\hat u)$, and not equal to $\lk_{\Gamma_G^e}(\hat u)$. This is a contradiction to Lemma~\ref{lemma:proper-in-star}, and completes our proof of the proposition.

Therefore, there only remains to prove the above claim.

\begin{proof}[Proof of Claim~\ref{claim:proper}]
We first observe that $\lk_{\Gamma_H^e}(u)\subseteq\st_{\Gamma_H^e}(\hat u)$. Indeed, every cyclic parabolic subgroup of $H$ commuting with $u$ must be contained in $H_u\times H_u^{\perp}$, and vertices in $\lk_{\Gamma_H^e}(u)$ are exactly the cyclic parabolic subgroups contained in $H_u^{\perp}$. Since $H_u^{\perp}\subseteq H_{\hat u}\times H_{\hat u}^{\perp}$, they all commute with $H_{\hat u}$, showing that $\lk_{\Gamma_H^e}(u)\subseteq\st_{\Gamma_H^e}(\hat u)$. 

In view of the above, we only need to prove that $\Gamma_H^{ue}$ contains a vertex in $\lk_{\Gamma^e_H}(\hat u)\setminus\lk_{\Gamma^e_H}(u)$. Arguing towards a contradiction, let us assume that it does not. 

We will now construct two full subgraphs $\Lambda_1,\Lambda_2\subseteq\lk_{\Gamma_H}(\hat u)$ such that $\lk_{\Gamma_H}(\hat u)=\Lambda_1\circ\Lambda_2$, with the following properties: every untransvectable vertex of $\Gamma_H$ which is contained in $\lk_{\Gamma_H}(\hat u)$ belongs to $\Lambda_1$, and $u$ is adjacent to every vertex in $\Lambda_1$. 

For this, let $\Lambda'_1=\lk_{\Gamma_H}(\hat u)\cap\lk_{\Gamma_H}(u)$ and $\Lambda'_2=\lk_{\Gamma_H}(\hat u)\setminus \lk_{\Gamma_H}(u)$ (note in particular that if $u$ and $\hat u$ are adjacent, then $u$ lies in $\Lambda'_2$). Let $\mathsf{U}$ be the set of all untransvectable vertices of $\Gamma_H$ that are contained in $\lk_{\Gamma_H}(\hat u)$. We notice that $\mathsf{U}\subseteq\Lambda'_1$: indeed otherwise, there would exist a vertex in $\mathsf{U}$ lying in $\lk_{\Gamma_H^e}(\hat u)\setminus\lk_{\Gamma_H^e}(u)$, contradicting our assumption. 

Now we define $\Lambda_1$ and $\Lambda_2$.
Denote by $(\Lambda'_1)^{\mathrm{op}}$ the \emph{opposite graph} of $\Lambda'_1$: it has the same vertex set as $\Lambda'_1$, and two vertices are adjacent in $(\Lambda'_1)^{\mathrm{op}}$ if and only if they are non-adjacent in $\Lambda'_1$. Note that $(\Lambda'_1)^{\mathrm{op}}$ might have multiple components.
Let $\Lambda_1$ be the full subgraph of $\Lambda'_1$ spanned by all vertices of $\Lambda'_1$ that can be connected to a vertex in $\mathsf{U}$ via an edge path in $(\Lambda'_1)^{\mathrm{op}}$.
We let $\Lambda_2$ be the full subgraph of $\Gamma$ spanned by $\Lambda'_2$ and all vertices in $\Lambda'_1\setminus \Lambda_1$. In particular $\lk_{\Gamma_H}(\hat u)$ is spanned by $\Lambda_1$ and $\Lambda_2$. 

\medskip

We claim that $\lk_{\Gamma_H}(\hat u)=\Lambda_1\circ\Lambda_2$. Indeed, let $x\in V\Lambda_1$ and $y\in V\Lambda_2$; we aim to prove that $x$ and $y$ are adjacent in $\Gamma_H$. Indeed:

$\bullet$ If $y=u$, then $x$ is adjacent to $y$ by definition of $\Lambda'_1$.

$\bullet$ If $y$ belongs to $\Lambda'_1\setminus \Lambda_1$, then it is adjacent to $x$ (otherwise, $y$ and $x$ are in the same path component of $(\Lambda'_1)^{\mathrm{op}}$).

$\bullet$ We finally assume that $y$ belongs to $V\Lambda'_2$ and is distinct from $u$, and assume towards a contradiction that $y$ and $x$ are not adjacent. By definition of $\Lambda_1$, there is a sequence of vertices $x=x_h,x_{h-1},\dots,x_0$ in $\Lambda_1$, where each $x_i$ is not adjacent to $x_{i-1}$ and $x_0\in \mathsf{U}$. For every $i\in\{0,\dots,h\}$, let $g_i$ be a generator of the cyclic subgroup $H_{x_i}$ associated to $x_i$, and let $g_y$ be a generator of the cyclic subgroup $H_y$ associated to $y$. Consider the cyclic group \[Z=g_yg_hg_{h-1}\dots g_2g_1 H_{x_0} g_1^{-1}g_2^{-1}\dots g_{h-1}^{-1}g_h^{-1}g_y^{-1},\] which is an untransvectable cyclic parabolic subgroup. Note that  any two consecutive letters appearing in this word correspond to non-adjacent vertices, and $u$ is non-adjacent to $y$ (as $y\in \Lambda'_2=\lk_{\Gamma_H}(\hat u)\setminus \lk_{\Gamma_H}(u)$). By considering the normal form in right-angled Artin groups (see \cite[Theorem~3.9]{Gre}, also \cite[Section~3]{HM}), we know that $Z_u$ does not commute with $Z$. On the other hand, since all vertices $x_i$ and $y$ belong to $\st_{\Gamma_H}(\hat u)$, the parabolic subgroup $Z_{\hat u}$ commutes with $Z$. Therefore $Z$ represents an untransvectable cyclic parabolic subgroup in $\lk_{\Gamma_H^e}(\hat u)\setminus\lk_{\Gamma_H^e}(u)$, a contradiction.
    
This proves our claim that $\lk_{\Gamma_H}(\hat u)=\Lambda_1\circ\Lambda_2$. And by construction, all vertices in $\mathsf{U}$ are in $\Lambda_1$, and $u$ is adjacent to all vertices in $\Lambda_1$.

\medskip

We now claim that all untransvectable cyclic parabolic subgroups that commute with $H_{\hat u}$, and distinct from $H_{\hat u}$, are of the form $hH_wh^{-1}$ with $h\in H_{\Lambda_1}$ and $w\in V\Lambda_1$. Indeed, let $P$ be such a subgroup, then $P$ is contained in the normalizer of $H_{\hat u}$, which is equal to  $H_{\hat u}\times H_{\hat u}^\perp=H_{\hat u}\times H_{\Lambda_1}\times H_{\Lambda_2}$. The type of $P$ is an untransvectable vertex $w$ adjacent to $\hat u$, so it belongs to $\Lambda_1$. So there exists $g\in H_{\hat u}\times H_{\hat u}^\perp$ such that $P=gH_wg^{-1}$ (the fact that $g$ can be chosen in $H_{\hat u}\times H_{\hat u}^{\perp}$ is a consequence of \cite[Proposition~2.2(2)]{CCV}). The above splitting of $H_{\hat u}\times H_{\hat u}^\perp$ implies that we can assume $g\in H_{\Lambda_1}$.

In particular, all untransvectable cyclic parabolic subgroups that commute with $H_{\hat u}$ also commute with $H_u$. This contradicts the fact that $\hat u$ is strongly untransvectable. This contradiction completes our proof.
\end{proof}
 \phantom\qedhere
\end{proof}

\section{Measure equivalence and orbit equivalence classification}\label{sec:end}

We now complete the proof of our main measure equivalence classification theorem (Theorem~\ref{theointro:main} from the introduction), whose statement we now recall.

\begin{theo}\label{theo:me-classification}
Let $G$ and $H$ be right-angled Artin groups, with $|\Out(G)|<+\infty$.

The groups $G$ and $H$ are measure equivalent if and only if there exists a finite-index right-angled Artin subgroup $G^0\subseteq G$ such that, denoting by $\Lambda$ the defining graph of $G^0$, the group $H$ is a graph product of infinite finitely generated free abelian groups over $\Lambda$.
\end{theo}

\begin{proof}
The conclusion is obvious if $G$ is trivial. If $G$ is infinite cyclic, then this is a consequence of Dye's theorem \cite{Dye1,Dye2} stating that all groups $\mathbb{Z}^n$ with $n\ge 1$ are orbit equivalent, and of the fact that an amenable group is never measure equivalent to a non-amenable group, see e.g.\ \cite[Corollary~3.2]{Fur-survey}. We now assume that $G$ is non-cyclic.

The ``if'' direction follows from the facts that $G$ is always measure equivalent to any finite-index subgroup of $G$, and that a graph product of countably infinite amenable groups over $\Lambda$ is always measure equivalent (in fact orbit equivalent) to the right-angled Artin group over $\Lambda$ by \cite[Proposition~4]{HH}.

Conversely, let us assume that $H$ is measure equivalent to $G$. Since $|\Out(G)|<+\infty$, the right-angled Artin group $G$ is clique-reduced. Using Remark~\ref{rk:clique-reduced}, we can write $H$ as a graph product of free abelian groups over a clique-reduced graph. Thus, in view of \cite[Proposition~4]{HH}, we do not lose any generality by assuming that $H$ is also clique-reduced to start with, and our new goal is to prove that $H$ embeds as a finite-index subgroup of $G$. 

By \cite[Corollary~9.4]{EH}, the untransvectable extension graphs $\Gamma_H^{ue}$ and $\Gamma_G^{ue}$ are isomorphic (and since $|\Out(G)|<+\infty$, the latter is the same as $\Gamma_G^e$). Since $|\Out(G)|<+\infty$, the graph $\Gamma_G$ does not contain any untransvectable non-abelian equivalence class, so by Proposition~\ref{prop:non-abelian-class}, neither does $\Gamma_H$. And every untransvectable vertex of $\Gamma_G$ is strongly untransvectable (Remark~\ref{rk:strongly-untransvectable}), so by Proposition~\ref{prop:strongly-untransvectable}, the same is true in $\Gamma_H$. The conclusion therefore follows from Proposition~\ref{prop:rigidity}.
\end{proof}

We finally complete the proof of our main orbit equivalence classification theorem (Theorem~\ref{theointro:oe} from the introduction), which we now recall.

\begin{theo}\label{theo:oe-classification}
Let $G,H$ be two right-angled Artin groups, with $|\Out(G)|<+\infty$.

The groups $G$ and $H$ are orbit equivalent if and only if $H$ is a graph product of infinite finitely generated free abelian groups over the defining graph of $G$.
\end{theo}

For the proof, recall that the \emph{coupling index} of a measure equivalence coupling $(\Omega,\mu)$ between two countable groups $G$ and $H$ is defined as \[[G:H]_\Omega=\frac{\mu(X_H)}{\mu(X_G)},\] where $X_G,X_H$ are respective Borel fundamental domains for the actions of $G$ and $H$.

\begin{proof}
The ``if'' direction follows from the fact that any graph product over $\Gamma_G$ with countably infinite amenable vertex groups is orbit equivalent to $G$ by \cite[Proposition~4]{HH}.

In view of Theorem~\ref{theo:me-classification}, in order to prove the ``only if'' direction, there only remains to see that $G$ is not orbit equivalent to any proper finite-index subgroup of $G$. So assume that $G$ is orbit equivalent to one of its finite-index subgroups $G^0$, and let us prove that $G^0=G$. We have two measure equivalence couplings bewteen $G$ and $G^0$, one of coupling index $1$ (given by the orbit equivalence), and one of coupling index $[G:G^0]$ (given by the action by left/right multiplication on $G$). By composition of couplings (as in \cite[Section~2]{Fur-me}), we derive a self measure equivalence coupling of $G$ of coupling index $[G:G^0]$. Since $|\Out(G)|<+\infty$, it follows from \cite[Proposition~5.9]{EH} that every self measure equivalence coupling of $G$ has coupling index $1$. So $[G:G^0]=1$, i.e.\ $G^0=G$, as desired.
\end{proof}

\small

\bibliographystyle{alpha}
\bibliography{classification-bib}

\begin{flushleft}
Camille Horbez\\ 
Universit\'e Paris-Saclay, CNRS,  Laboratoire de math\'ematiques d'Orsay, 91405, Orsay, France \\
\emph{e-mail:~}\texttt{camille.horbez@universite-paris-saclay.fr}\\[4mm]
\end{flushleft}

\begin{flushleft}
Jingyin Huang\\
Department of Mathematics\\
The Ohio State University, 100 Math Tower\\
231 W 18th Ave, Columbus, OH 43210, U.S.\\
\emph{e-mail:~}\texttt{huang.929@osu.edu}\\
\end{flushleft}

\vfill

\noindent \textcolor{black!70}{This work is openly licensed via Creative Commons CC-BY 4.0}\\
\href{https://creativecommons.org/licenses/by/4.0/}{https://creativecommons.org/licenses/by/4.0/}.

\end{document}